\documentclass[11pt]{article}

\usepackage[english]{babel}
\usepackage[utf8x]{inputenc}
\usepackage[T1]{fontenc}

\usepackage[a4paper,top=2.75cm,bottom=2.75cm,left=2.75cm,right=2.75cm,marginparwidth=1.75cm]{geometry}
\usepackage{amsmath}
\usepackage{amsthm}
\usepackage{amssymb}
\usepackage{bbm}
\usepackage{IEEEtrantools}
\usepackage[]{mathrsfs}
\usepackage{siunitx}
\usepackage{graphicx}
\usepackage[colorinlistoftodos]{todonotes}
\usepackage[colorlinks=true, allcolors=blue]{hyperref}
\usepackage{bm}
\usepackage{algorithm}
\usepackage{algorithmic}
\usepackage{enumitem}

\usepackage{booktabs,multirow,array}
\newcommand{\tabitem}{~~\llap{\textbullet}~~}

\renewcommand{\algorithmicrequire}{\textbf{Input:}}
\renewcommand{\algorithmicensure}{\textbf{Output:}}

\theoremstyle{definition}
\newtheorem{definition}{Definition}[section]

\newtheorem{proposition}{Proposition}[section]

\newtheorem{lemma}{Lemma}[section]
\newtheorem{theorem}{Theorem}[section]

\newtheorem{remark}{Remark}[section]

\newcommand{\one}{\mathbf{1}}

\newcommand{\h}{\mathbf{h}}

\newcommand{\hP}{\hat{P}}
\newcommand{\tT}{\tilde{T}}
\newcommand{\obs}{\text{obs}}
\newcommand{\mle}{\text{mle}}
\newcommand{\KL}{\text{KL}}

\newcommand{\mE}{\mathbb{E}}
\def\tE{{\text{E}}}

\title{
\textbf{Maximum Likelihood Estimation for\\ Learning Populations of Parameters}}
\author{Ramya Korlakai Vinayak$^\dagger$, Weihao Kong$^\ddagger$, Gregory Valiant$^\ddagger$, Sham M. Kakade$^{\dagger}$\\ \\
$^\dagger$Allen School of Computer Science \& Engineering, University of Washington\\
$^\ddagger$Department of Computer Science, Stanford University\\
\tt{\{ramya, sham\}@cs.washington.edu, \{whkong, valiant\}@stanford.edu}}
\begin{document}
\date{}
\maketitle

\begin{abstract}
Consider a setting with $N$ independent individuals, each with an unknown parameter, $p_i \in [0, 1]$ drawn from some unknown distribution $P^\star$. After observing the outcomes of $t$ independent Bernoulli trials, i.e., $X_i \sim \text{Binomial}(t, p_i)$ per individual, our objective is to accurately estimate $P^\star$. This problem arises in numerous domains, including the social sciences, psychology, health-care, and biology, where the size of the population under study is usually large while the number of observations per individual is often limited. 

Our main result shows that, in the regime where $t \ll N$, the maximum likelihood estimator (MLE) is both statistically minimax optimal and efficiently computable. Precisely, for sufficiently large $N$, the MLE achieves the information theoretic optimal error bound of $\mathcal{O}(\frac{1}{t})$ for $t < c\log{N}$, with regards to
the earth mover's distance (between the estimated and true distributions). More generally, in an exponentially large interval of $t$ beyond $c \log{N}$, the MLE achieves the minimax error bound of $\mathcal{O}(\frac{1}{\sqrt{t\log N}})$. In contrast, regardless of how large $N$ is, the naive "plug-in" estimator for this problem only achieves the sub-optimal error of $\Theta(\frac{1}{\sqrt{t}})$.
\end{abstract}

\section{Introduction}
\label{sec:intro}
The problem of learning a distribution of parameters over a population arises in several domains such as social sciences, psychology, medicine, and biology~\cite{lord1965strong, lord1975empirical, millar1986distribution, palmer1990small, colwell1994estimating, bell2000environmental}. While the number of individuals in the population can be very large, the number of observations available per individual is often very limited, which prohibits accurate estimation of the parameter of interest per individual. In such sparse observation scenarios, \emph{how accurately can we estimate the distribution of parameters over the population?} 

In the 1960's F. M. Lord studied the problem of estimating the distribution of parameters over a population in the context of psychological testing~\cite{lord1965strong, lord1969estimating}. Consider a study involving a large number of independent individuals. Each individual has an unknown probability $p_i$ of answering a question correctly. Given the scores of these individuals on a test with a small set of questions, the goal is to estimate the underlying distribution of the $p_i$'s. Such an estimated distribution can be used in downstream tasks, like testing if the distribution of scores is uniform or multimodal, or comparing two tests of the same psychological trait.

We use the lens of sparse regime analysis for this problem of learning a population of parameters. Our analysis is inspired by the recent advances in a related problem of estimating discrete distributions and their properties such as, entropy and support size, when the number of observations is much smaller than the support size of the distribution~\cite{valiant2011estimating, valiant2011power, jiao2015minimax, wu2015chebyshev, valiant2016instance, wu2016minimax, orlitsky2016optimal, acharya2017unified, jiao2018minimax, Han2018}. However, we note that our setting is not the same as estimating a discrete distribution. For instance, the probabilities sum to 1 for a discrete distribution, where as, the true parameters in our setting need not sum to 1.

There have been several classical works on non-parametric mixture models in general~\cite{turnbull1976empirical, simar1976maximum, laird1978nonparametric, lindsay1983geometry, lindsay1983geometry2, bohning1989likelihood, lesperance1992algorithm} and binomial mixture models in particular~\cite{cressie1979quick, wood1999binomial} which have studied the geometry of the maximum likelihood estimator (MLE), the optimality conditions, identifiability, and uniqueness of the MLE solution, and algorithms for computing the optimal solution to the MLE. However, the statistical analysis of how accurately the MLE recovers the underlying distribution has not been addressed. In this paper, we fill this gap, and show that the MLE achieves the optimal error bound with regards to the earth mover's distance (or Wasserstein-1 distance, Definition~\ref{def:W1}) between the estimated and true distributions (equivalently, the $l_1$-distance between the CDF's).

\subsection{Problem set-up and summary of results}
\label{sec:problemsetup}
The setting considered in~\cite{lord1969estimating} can be modeled as follows. Consider a set of $N$ independent coins, each with its own unknown bias $p_i\in [0, 1]$ drawn independently from some unknown distribution $P^\star$ over $[0, 1]$. That is, the probability of seeing a head when coin $i$ is tossed is $p_i$. For each coin $i$, we get to observe the outcome of $t$ independent tosses, denoted by, $X_i \sim \text{Binomial}(t, p_i)$. Our goal is to estimate the unknown distribution $P^\star$ from $\{X_i\}_{i=1}^N$.

The MLE for this problem is formulated as follows:
\begin{equation*}
\hP_{\mle} \in \underset{Q \in \mathcal{D}}{\text{arg max}}\ \sum_{i=1}^N \log{ \int_{0}^1 \binom{t}{X_i} y^{X_i} (1 - y)^{t-X_i} dQ(y)},
\end{equation*}
where $\mathcal{D}$ is the set of all distributions on $[0, 1]$. 

\textbf{Our Contribution:} We bound the earth mover's distance (or the Wasserstein-1 distance) between the true distribution $P^\star$ and the MLE solution $\hP_{\mle}$, and show that:
\begin{theorem} (Informal statement)

\begin{itemize}[leftmargin=*, topsep=0pt]
    \item The MLE achieves an error bound of 
    \begin{equation*}
        W_1(P^\star, \hP_{\mle}) = \mathcal{O}_{\delta}\left( \frac{1}{t} \right),\footnote{$\mathcal{O}_\delta(.)$ hides $\log{(1/\delta)}$ in the bound for it to hold with probability at least $1-2\delta$.}
    \end{equation*} when $t = \mathcal{O}(\log{N})$. The bound of $\Theta\left( \frac{1}{t} \right)$ is information theoretically optimal up to a constant factor.
    \item  The MLE achieves an error bound of 
     \begin{equation*}W_1(P^\star, \hP_{\mle}) = \mathcal{O}_\delta\left(\frac{1}{\sqrt{t \log{N}}} \right),
     \end{equation*}
     when $ t \in \left[ \Omega(\log{N}), \mathcal{O}\left({N^{2/9 - \epsilon}} \right) \right]$, and this bound is information theoretically optimal in this regime. 
   \end{itemize}
   \end{theorem}
\iffalse
\begin{theorem}
Consider a set of $N$ independent coins, each with its own unknown bias $p_i\in [0, 1]$ drawn independently from some unknown distribution $P^\star$ over $[0, 1]$. For each coin $i$, we get to observe the outcome of $t$ independent tosses, denoted by, $X_i \sim \text{Binomial}(t, p_i)$. For all $t=\mathcal{O}\left(N^{2/9 - \epsilon} \right)$, with probabilty $1-\delta$, the maximum likelihood estimator of $P$, defined as
\begin{equation*}
\hP_{\mle} \in \underset{Q \in \mathcal{D}}{\text{arg max}}\ \sum_{i=1}^N \log{ \int_{0}^1 \binom{t}{X_i} y^{X_i} (1 - y)^{t-X_i} dQ(y)},
\end{equation*}
where $\mathcal{D}$ is the set of all distributions on $[0, 1]$, satisfies
\begin{equation*}W_1(P^\star, \hP_{\mle}) = \mathcal{O}_\delta\left(\max\left(\frac{1}{\sqrt{t \log{N}}}, \frac{1}{t}\right)\right).
     \end{equation*}
\end{theorem}
\fi
Table~\ref{table:compareresults} summarizes our results in comparison to other estimators. While the moment matching estimator~\cite{Tian2017} achieves the same minimax optimal error bound as the MLE when $t = \mathcal{O}(\log{N})$, it fails when $t = \Omega(\log{N})$ due to high variance in the larger moments. While the local moment matching approach~\cite{Han2018} could theoretically avoid this weakness, it involves hyperparameter tuning which makes it difficult to work with in practice (Remark~\ref{remark:localmoments} in Section~\ref{sec:mainresults}). In contrast, the MLE naturally adapts itself and achieves the optimal rates in different regimes without the need for any parameter tuning. 
Furthermore, our analysis involves bounding the coefficients of Bernstein polynomials approximating Lipschitz-1 functions (Proposition~\ref{prop:coeffbound}). This question is of independent interest with implications to general polynomial approximation theory as well as applications in computer graphics. 

\subsection{Outline}
The rest of the paper is organized as follows. In Section~\ref{sec:relatedworks}, we discuss the related works. In Section~\ref{sec:mainresults}, we describe the maximum likelihood estimator for the problem and formally state our main results.  We provide the outline of the proofs of our main results in Section~\ref{sec:proofsketches}. The details of the proofs are available in the appendix. 
Finally, we conclude in Section~\ref{sec:conclusion} by discussing some open questions for future research directions.

\begin{table}[]
\caption{Comparison of results}\label{table:compareresults}
\begin{center}
\begin{tabular}{|c|c|}
%\begin{tabular}{|m{0.275\columnwidth}|m{0.55\columnwidth}|}
\hline 
&  \\
\textbf{Estimators} & \textbf{Bound on EMD}\\[5pt]
\hline
&  \\
 \shortstack{Empirical\\ Estimator} & \shortstack{$\Theta\left(\frac{1}{\sqrt{t}}\right) + \Theta\left(\frac{1}{\sqrt{N}}\right)$\\ in all regimes} \\[10pt]
\hline
&  \\
\shortstack{Moment\\ Matching\\ \cite{Tian2017}} &  \shortstack{\tabitem$\Theta\left( \frac{1}{t}\right) $,  when $t = \mathcal{O}(\log{N})$\\ \tabitem Fails when $t = \Omega(\log{N})$ }   \\ [10pt]
\hline
&  \\
\shortstack{MLE\\ (this paper)} & \shortstack{\tabitem $\Theta\left( \frac{1}{t}\right)$, when $t = \mathcal{O}(\log{N})$\\ \tabitem $\Theta\left( \frac{1}{\sqrt{t\ \log{N}}}\right), $ when $t \in \left[\Omega(\log{N}),\ \mathcal{O}\left({N^{2/9 - \epsilon}} \right)\right]$ }   \\[10pt]
\hline 
\end{tabular}
\end{center}
\end{table}
\section{Related Works}
\label{sec:relatedworks}
Starting from~\cite{lord1969estimating}, there has been a great deal of interest in the problem of estimating the distribution of true scores of a population of independent entities. Maximum likelihood estimation for non-parametric mixture models has been studied extensively~\cite{lord1975empirical, cressie1979quick, laird1978nonparametric, turnbull1976empirical, lesperance1992algorithm}. \cite{lindsay1983geometry} and \cite{lindsay1983geometry2} delineate the geometry of the MLE landscape for non-parametric mixture models in general, and specifically for exponential family respectively. \cite{wood1999binomial} further discusses the issue of uniqueness of the solution for mixture of binomials and the relationship with the moment space. As mentioned in the introduction, the accuracy of the MLE solution for this formulation has not been studied in the literature. Our work fills in this gap by showing that the MLE solution is minimax optimal when $t \ll N$.

In a recent work~\cite{Tian2017}, the authors proposed a \emph{moment matching estimator} to estimate the unknown distribution of the biases in the regime where the number of tosses per coin $t = \mathcal{O}(\log{N})$. This estimator finds a distribution on $[0, 1]$ that closely matches the first $t$ empirical moments of the unknown distribution that can be estimated using the observations. This moment matching estimator has an error bound of $\mathcal{O}\left( \frac{1}{t} \right) + \mathcal{O}_\delta\left( 2^t t \sqrt{\frac{\log{t}}{N}}  \right)$ in Wasserstein-1 distance. Furthermore, ~\cite{Tian2017} also showed that $\Omega\left(\frac{1}{t}\right)$ is a lower bound in this setting.  The main weakness of this method of moments approach is that it fails to obtain the optimal rate when $t > c \log{N}$. 

A tangentially related problem is that of estimating a discrete distribution and its symmetric properties\footnote{A function over a discrete distribution is said to be a \emph{symmetric function} if it remains invariant to the relabeling of the domain symbols.} such as, entropy and support size, when the number of observations is much smaller than the support size of the distribution. This is a well-studied classical problem in statistics~\cite{fisher1943relation, good1956number, efron1976estimating}. It has received a lot of interest in the past decade and continues to be a very active area of research~\cite{paninski2003estimation, orlitsky2004modeling, acharya2009recent, acharya2010exact, valiant2011estimating, valiant2013estimating,jiao2015minimax, wu2015chebyshev, valiant2016instance, wu2016minimax, orlitsky2016optimal, acharya2017unified, jiao2018minimax, Han2018}. Recent work~\cite{Han2018} used local moment matching to provide bounds on estimating symmetric properties of discrete distributions under the Wasserstein-1 distance. This technique of local moment matching can be used in our setting to improve the bounds obtained in~\cite{Tian2017} in the regime where $t > c \log N$. We discuss this more in Section~\ref{sec:mediumsample}.

In a similar spirit to our work, a series of works~\cite{acharya2009recent,acharya2010exact,acharya2017unified} examined the \textit{profile} or \textit{pattern} maximum likelihood as a unifying framework for estimating symmetric properties of a discrete distribution. Unlike in our setting, it is computationally challenging to compute the exact maximum likelihood estimator, and the question becomes how to efficiently approximate it (see e.g.~\cite{vontobel2012bethe}).
\section{Main Results}
\label{sec:mainresults}
Before formally stating our results, we introduce some notation, discuss the MLE objective and define the Wasserstein-1 metric used to measure the accuracy of estimation.\\

\noindent \textbf{Notation:} Recall that $N$ is the number of independent coins and $t$ is the number of tosses per coin. The biases of the coins are denoted by $\{p_i\}_{i=1}^N$, where each $p_i\in [0, 1]$ is drawn from some unknown distribution $P^\star$ on $[0, 1]$. The set of observations is $\{X_i\}_{i=1}^N$, where $X_i \sim \text{Binomial}(t, p_i)$.
For $s \in \{0, 1, ..., t\}$, let $n_s$ denote the number of coins that show $s$ heads out of $t$ tosses. Let $h_s^{\obs}$ denote the fraction of coins that show $s$ heads. 
\begin{equation}
     n_s := \sum_{i=1}^N \one_{\{X_i = s\}},\ h_s^{\obs} := \frac{n_s}{N},% \label{eqn:def-ns-hs}
\end{equation}
where $\one_{\mathcal{A}}$ is indicator function for set $\mathcal{A}$. $\h^{\obs} := \{h_0^{\obs}, h_1^{\obs}, ..., h_t^{\obs} \}$ is the observed \emph{fingerprint}. Since the identity of the coins is not important to estimate the distribution of the biases, the observed fingerprint is a sufficient statistics for the estimation problem.\\

\noindent \textbf{MLE Objective:} The MLE estimate of the distribution of biases given the observations $\{X_i\}_{i=1}^N$ is,
\begin{IEEEeqnarray*}{rCl}
\hP_{\mle} &\in& \underset{Q \in \mathcal{D}}{\text{arg max}}\ \sum_{i=1}^N \log{ \int_{0}^1 \binom{t}{X_i} y^{X_i} (1 - y)^{t-X_i} dQ(y)},\\
&=& \underset{Q \in \mathcal{D}}{\text{arg max}} \sum_{s=0}^t n_s \log{ \underset{=:E_Q[h_s]}{\underbrace{\int_{0}^1 \binom{t}{s} y^{s} (1 - y)^{t -s} q(y) dy } }},
\end{IEEEeqnarray*}
where $\mathcal{D}$ is the set of all distributions on $[0, 1]$, $n_s$ is the number of coins that that see $s$ heads out of $t$ tosses, and $E_Q[h_s]$ is the expected fraction of the population that sees $s$ heads out of $t$ tosses under the distribution $Q$. Equivalently, the MLE can be written in terms of the fingerprint as follows,
\begin{IEEEeqnarray}{rCl}
\hP_{\mle} &\in& \underset{Q \in \mathcal{D}}{\text{arg max}}\ \sum_{s=0}^t h_s^{\text{obs}} \log{E_Q[h_s]},\label{eqn:MLEobj} \\
&=& \underset{Q \in \mathcal{D}}{\text{arg min}}\ \text{KL}\left( \mathbf{h}^{\obs},E_{Q}[\h] \right), \label{eqn:MLEobjKL}
\end{IEEEeqnarray}
where $\text{KL}(A, B)$ is the Kullback-Leibler divergence\footnote{KL divergence between two discrete distributions A and B supported on $\mathcal{X}$ is defined as $\text{KL}(A, B) = \sum_{x \in \mathcal{X}} A(x) \log{\frac{A(x)}{B(x)}}$.} between distributions $A$ and $B$, $\h^{\obs}$ is the observed fingerprint vector and $E_{Q}[\h]$ denotes the expected fingerprint vector when the biases are drawn from distribution $Q$.
\begin{remark}\label{remark:MLEGeometry}
The set $\mathcal{D}$ of all distributions over $[0, 1]$ is convex. Furthermore, the objective function of the MLE (Equation~\ref{eqn:MLEobjKL}) is convex in $Q$ and strictly convex in the valid fingerprints, $\{E_Q[h_s]\}_{s=0}^t$. While there is a unique $E_{\hP_{\mle}}[\h]$ that minimizes the objective~\eqref{eqn:MLEobjKL}, there can be many distributions $Q^\star \in \mathcal{D}$ that can give rise to the optimal expected fingerprint. Moreover, while the fingerprint vector $\h$ lives in $\Delta^t$, the $t$-dimensional simplex in $\mathbb{R}^{t+1}$, not all vectors in $\Delta^t$ can be valid fingerprints. The set of all valid fingerprints is a small convex subset of $\Delta^t$. Very often $\h^{\obs}$ falls outside the set of valid fingerprints and the solution to the MLE is the closest projection under the KL divergence onto the valid fingerprint set. Furthermore, the fingerprints are related to moments via a linear transform. The geometry of the set of valid fingerprints therefore can also be described using moments. For more details on this geometric description we refer the reader to~\cite{wood1999binomial}.
\end{remark}

\noindent \textbf{Wasserstein-1 Distance:} We measure the accuracy of our estimator using the Wasserstein-1 distance or the earth mover's distance (EMD) between two probability distributions over the interval $[0, 1]$ which is defined as:
\begin{definition}[Wasserstein-1 or earth mover's distance]\label{def:W1}
\begin{equation}
W_1(P, Q) := \underset{\gamma \in \Gamma(P, Q)}{\text{inf}} \int_{x=0}^1 \int_{y=0}^1 |x-y|\ d\gamma(x, y),
\end{equation}
where $\Gamma(P, Q)$ is a collection of all the joint distributions on $[0, 1]^2$ with marginals P and Q. A dual definition due to Kantarovich and Rubinstein~\cite{Kantorovich1958} of this metric is as follows:
\begin{IEEEeqnarray}{rCl}
\label{eqn:EMD}
W_1(P, Q) &:=& \underset{f \in \text{Lip}(1)}{\text{sup}} \int_0^1 f(x) (p(x) - q(x)) dx, \\
&=& \underset{f \in \text{Lip}(1)}{\text{sup}} \left( \text{E}_P[f] - \text{E}_Q[f] \right),
\end{IEEEeqnarray}
where $p$ and $q$ are the probability density functions of the distributions P and Q respectively, and $\text{Lip}(1)$ denotes the set of Lipschitz-1 functions. 
\end{definition}
Wassertein-1 distance is a natural choice to measure the accuracy of an estimator in our setting. E.g., suppose the true distribution $P^\star$ is $\delta(0.5) = 1$. Let $P_1$ with $\delta(0.45) = 1$ and $P_2 $ with $\delta(0) = \delta(1) = \frac{1}{2}$ be the output of two estimators. The Wassertein-1 distance, $W_1(P^\star, P_1) = 0.05$ and $W_1(P^\star, P_2) = 0.5$, clearly distinguishes the first estimate to be much better than the second. In contrast, the total variation distance between both $P_1$ and $P_2$ to the truth is $1$ and the KL divergence to the truth in both cases is infinite. 

\subsection{Small sample regime}
We first focus on the regime where the number of observations per coin, $t = \mathcal{O}(\log{N})$. Consider the problem setup in Section~\ref{sec:problemsetup}. The following theorem gives a bound on the Wasserstein-1 distance between the MLE (Equation~\ref{eqn:MLEobjKL}) and the true underlying distribution.
\begin{theorem}[Small Sample Regime]
\label{thm:smalltMLE}
When $t = \mathcal{O}(\log{N})$, the Wassertein-1 distance between an optimal solution to the MLE, denoted by $\hP_{\mle}$ and the true underlying distribution $P^\star$ can be bounded with probability at least $1-2\delta$ as follows,
\begin{equation}
\label{eqn:EMDSmallt}
    W_1(P^\star, \hP_{\mle}) \leq \mathcal{O}_{\delta}\left( \frac{1}{t} \right).
\end{equation}
\end{theorem}
For constant $\delta$, this $O\left(\frac{1}{t}\right)$ rate is information theoretically optimal due to the following result (Proposition 1 in ~\cite{Tian2017}):
\begin{proposition}[Lower Bound~\cite{Tian2017}]
Let $P$ denote a distribution over $[0, 1]$. Let $\mathbf{X} := \{X_i\}_{i=1}^N$ be random variables with $X_i \sim \text{Binomial}(t, p_i)$ where $p_i$ is drawn independently from $P$. Let $f$ be an estimator that maps $\mathbf{X}$ to a distribution $f(\mathbf{X})$. For every fixed $t$, the following lower bound holds for all $N$:
\begin{equation}
    \underset{f}{\text{inf}}\ \underset{P}{\text{sup}}\ \text{E}\left[W_1(P, f(\mathbf{X}))\right] > \frac{1}{4t}.
\end{equation}
\end{proposition}

\subsection{Medium sample regime}
\label{sec:mediumsample}
In this section we consider the regime where the number of observations per coin $t$ is greater than $\Omega(\log{N})$. For the same setting as before (Section~\ref{sec:problemsetup}), the following theorem provides a bound on the Wasserstein-1 distance between the MLE solution and the true distribution.
\begin{theorem}[Medium Sample Regime]
\label{thm:mediumtMLE}
There exists $\epsilon > 0$, such that, for\\ $t \in \left[\Omega(\log{N}) ,\  \mathcal{O}\left({N^{2/9 - \epsilon}} \right) \right]$, with probability at least $1-2\delta$,
\begin{equation}
\label{eqn:EMDMediumt}
    W_1(P^\star, \hP_{\mle}) \leq \mathcal{O}_\delta\left(\frac{1}{\sqrt{t \log{N}}} \right) .
\end{equation}
\end{theorem}
We prove a matching $\Omega(\frac{1}{\sqrt{t\log N}})$ lower bound on the minimax rate for estimating the population of parameters under Wasserstein-1 distance. The lower bound is formalized in the following theorem.
\begin{theorem}\label{thm:lowerboundtlogN}
Let $P$ be a distribution over $[0, 1]$. Let $\mathbf{X} := \{X_i\}_{i=1}^N$ be random variables with $X_i \sim \text{Binomial}(t, p_i)$ where $p_i$ is drawn independently from $P$. Let $f$ be an estimator that maps $\mathbf{X}$ to a distribution $f(\mathbf{X})$. For every $t, N$ s.t. $t\le \frac{N^{2(e^4-1)}}{36}$, the following lower bound holds:
\begin{equation}
    \underset{f}{\text{inf}}\ \underset{P}{\text{sup}}\ \text{E}\left[W_1(P, f(\mathbf{X}))\right] > \frac{1}{3e^4\sqrt{t\log N}}.
\end{equation}
\end{theorem}
This lower bound, combined with the $\Theta(\frac{1}{t})$ lower bound shown in~\cite{Tian2017}, implies that the MLE is minimax optimal up to a constant factor in both the regimes. 

\begin{remark} [\textbf{Conjecture}]
We believe that the range of $t$ for which the bound in Equation~\ref{eqn:EMDMediumt} holds is larger than that guaranteed in Theorem~\ref{thm:mediumtMLE}. With the current proof framework, it seems likely that the interval of $t$ in which Theorem~\ref{thm:mediumtMLE} holds can be improved to $$t \in \left[\Omega(\log{N}),\  \mathcal{O}\left({N^{2/3 - \epsilon}} \right)\right].$$
Details on why we believe that this interval should hold are described in Section~\ref{sec:coeffboundconj}.
\end{remark}

\begin{remark} [\textbf{Local Moment Matching}]\label{remark:localmoments} The moment matching estimator in~\cite{Tian2017} fails when  $t$ is larger than $\Omega(\log N)$ because the $t$-th order moments cannot be estimated accurately in that regime. This causes the second term in the error bound $\mathcal{O}\left( \frac{1}{t} \right) + \mathcal{O}_\delta\left( 2^t t \sqrt{\frac{\log{t}}{N}}  \right)$ to become large. Naturally, one might consider matching only the first $\log N$ moments which can be reliably estimated. In addition, the parameter interval $[0,1]$ can be split into blocks, and the moment matching can be done in each block locally by utilizing the fact that for large $t$, $X_i/t$ tightly concentrates around $p_i$. The local moment matching was first introduced in a recent work by~\cite{Han2018} in the setting of learning discrete distributions.
Potentially, one may apply the local moment matching approach to our setting of learning populations of parameters which will likely yield an algorithm that achieves the same Wasserstein-1 distance error as the MLE, $\mathcal{O}(\max(\frac{1}{\sqrt{t\log N}},\frac{1}{t}))$ in the $t\ll N$ regime. The algorithm will degenerate to the one developed in~\cite{Tian2017} in the $t=\mathcal{O}(\log N)$ regime.
However, from a practical perspective, the local moment matching algorithm is quite unwieldy. It involves significant parameter tuning and special treatment for the edge cases. Some techniques used in local moment matching, e.g. using a fixed blocks partition of $[0,1]$ and matching the first $\log N$ moments for all the blocks, are quite crude and likely lose large constant factors both in theory and in practice. Therefore, we expect the local moment matching to have inferior performance than the MLE approach in practice. We include a brief sketch of how one may apply the local moment matching approach to our setting in Appendix~\ref{app:localmomentmatching}.
\end{remark}

\begin{remark}[\textbf{Empirical Estimator}]
The naive ``plug-in'' estimator for the underlying distribution is the sorted estimates of the biases of the coins. This incurs an error of $\mathcal{O}\left(\frac{1}{\sqrt{t}}\right) + \mathcal{O}\left(\frac{1}{\sqrt{N}}\right)$ in the earth movers distance (or $l_1-$ distance between the estimated and the true CDFs), where the first term is due to the error in estimating the biases of the coins from $t$ outcomes, and the second term is due to estimating the error in the estimated CDF using $N$ coins. If the number of tosses per coin is very large, that is, $t \gg N$, then we can estimate individual biases pretty well, and obtain an empirical CDF that can estimate $P^\star$ incurring an overall error rate of  $\mathcal{O}\left(\frac{1}{\sqrt{N}}\right)$. However, in the regime of interest, the number of observations per coin is small,  i.e., $t \ll N$ (\emph{sparse regime}). The empirical estimates of the biases in this regime are very crude. Thus, when $t$ is small, even with a very large population (large $N$), the empirical estimator does not perform better on the task of estimating the underlying distribution than on estimating the biases itself which incurs a $\Theta\left(\frac{1}{\sqrt{t}}\right)$ error.
\end{remark}
\section{Proof Sketches}
\label{sec:proofsketches}
In this section we provide proof sketches for the main results stated in Section~\ref{sec:mainresults}. The details are provided in the appendix.
\subsection{Bound on Wasserstein-1 distance}
Proofs of Theorems~\ref{thm:smalltMLE} and~\ref{thm:mediumtMLE} involve bounding the Wasserstein-1 distance between the true distribution $P^\star$ and the MLE estimate $\hP_{\mle}$.
Recall the dual definition of Wasserstein-1 distance or the earth movers distance between two distributions $P$ and $Q$ supported on $[0, 1]$,
\begin{equation*}
W_1(P, Q) = \underset{f \in \text{Lip}(1)}{\text{sup}} \int_{0}^1 f(x) (p(x) - q(x)) dx,  
\end{equation*}
where $p$ and $q$ are the probability density functions of the distributions $P$ and $Q$ respectively, and $\text{Lip}(1)$ denotes the set of Lipschitz-1 functions.
Any Lipschitz-1 function $f$ on $[0, 1]$ can be approximated using Bernstein polynomials as, $\hat{f}(x) := \sum_{j=0}^t b_j \binom{t}{j} x^j (1-x)^{t-j}$.
Using this approximation for any Lipschitz-1 function $f$, we obtain the following bound,
\begin{IEEEeqnarray}{rCl}
\int_{0}^1 f(x) (p(x) - q(x)) dx 
&=&  \int_{0}^1 \left( f(x) -\hat{f}(x) \right)(p(x) - q(x)) dx +  \int_{0}^1 \hat{f}(x) (p(x) - q(x)) dx,\nonumber \\
&\leq& 2 ||f - \hat{f}||_\infty  
+ \int_{0}^1 \sum_{j=0}^t b_j \binom{t}{j} x^j (1-x)^{t-j} (p(x) - q(x)) dx,   \nonumber \\
&=& 2 ||f - \hat{f}||_\infty +  \sum_{j=0}^t b_j \left( \text{E}_P[h_j] - \text{E}_Q[h_j] \right),
\label{eqn:boundEMDIntegralFingerprint}
\end{IEEEeqnarray}
where $||f - \hat{f}||_\infty := \underset{x \in [0, 1]}{\max}|f(x) - \hat{f}(x)|$ is the polynomial approximation error.
Therefore, the Wasserstein-1 distance (Definition~\ref{eqn:EMD}) between the true distribution $P^\star$ and MLE estimate $\hP_{\mle}$ can be bounded as follows,
\begin{IEEEeqnarray}{rCl}
&&W_1(P^*, \hP_{\mle}) \nonumber\\
&&\leq \underset{f \in \text{Lip}(1)}{\text{sup}} \left\{ 2 \underset{(a)}{\underbrace{||f - \hat{f}||_\infty}} 
+ \underset{(b)}{\underbrace{ \sum_{j=0}^t b_j \left( \text{E}_{P^*}[h_j] - h_j^{\text{obs}} \right)}} 
 +  \underset{(c)}{\underbrace{\sum_{j=0}^t b_j \left( h_j^{\text{obs}} - \text{E}_{\hP_{\mle}}[h_j] \right)}}  \right\} \label{eqn:boundEMDFingerprintEst}
\end{IEEEeqnarray}
The first term $(a)$ in the above bound (Equation~\ref{eqn:boundEMDFingerprintEst}) is the approximation error for using Bernstein polynomials to approximate Lipschitz-1 functions. The second term $(b)$ is the error due to sampling. The third term $(c)$ is the estimation error in matching the fingerprints.

\subsection{Concentration of fingerprints}
\label{sec:concFingerprints}
We bound the second term in Equation~\ref{eqn:boundEMDFingerprintEst} using the following lemma.
\begin{lemma}\label{lemma:concfingerprints}
With probability at least $1-\delta$, 
\begin{IEEEeqnarray}{rCl}
 \left| \sum_{j=0}^t b_j \left(h_j^{\text{obs}} - \text{E}_{P^*}[h_j] \right) \right| \leq \mathcal{O}\left( \max_j |b_j| \sqrt{\frac{\log{1/\delta}}{N}}\right).
\end{IEEEeqnarray}
\end{lemma}
\begin{proof}
Recall that $h_s^{\text{obs}}$ is the fraction of the population that sees $s$ heads out of $t$ tosses and $\text{E}_{P^*}[h_s]$ is the expected fingerprint under the true distribution which is exactly $\text{E}[h_s^{\text{obs}}]$, and we will use $\text{E}[h_s^{\text{obs}}]$ and $\text{E}_{P^*}[h_s]$ interchangeably. Define, $\phi(X) := \sum_{s=0}^t b_s \left( h_s^{\text{obs}} - \text{E}[h_s^{\text{obs}}] \right)$, that is,
\begin{IEEEeqnarray*}{rCl}
    \phi(X) = \frac{1}{N}\sum_{i=1}^N \sum_{s=0}^t b_s \left( \mathbbm{1}_{\{X_i = s\}} - \text{E}[h_s^{\text{obs}}] \right).%\label{eqn:fingerprint}
\end{IEEEeqnarray*}
Note that $\text{E}[\phi(X)] = 0$. Note that $X_i$ is a sum of $t$ independent Bernoulli random variables,
\[ X_i := \sum_{s=0}^t Y_s^{(i)}, \]
where $Y_s^{(i)} \sim \text{Bernoulli}(p_i)$. 
Let $\phi_{i'}(X)$ be $\phi$ with one of the $t$  tosses of coin $i$ being re-drawn, say $Y_s^{(i)'} \sim \text{Bernoulli}(p_i)$. Let $X_{i}^{'}$ denote the number of heads seen by coin $i$ with the replaced coin toss. Note that $X_{i}^{'}$ can differ from $X_i$ by at most $\pm 1$. Therefore, we can bound the difference in $\phi(X)$ and $\phi_{i'}(X)$ as follows,
\[ | \phi(X) - \phi_{i'}(X)  | \leq \underset{0 \leq s \leq t}{\max}\ |b_s| \frac{2}{N}.  \]
By McDiarmid's Inequality, for some absolute constants $C,  c > 0$,
$$
    \text{Pr}\left( | \phi(X) |  \geq \epsilon \right) 
 \leq  C\exp{\left( - \frac{c N \epsilon^2}{4 \left(\max_{s} |b_s|\right)^2} \right)}. 
$$
Hence, with probability at least $1-\delta$, 
\begin{IEEEeqnarray}{rCl}
 \left| \sum_{j=0}^t b_j \left(h_j^{\text{obs}} - \text{E}_{P^*}[h_j] \right) \right| \leq \mathcal{O}\left( \max_j |b_j| \sqrt{\frac{\log{1/\delta}}{N}}\right).
\end{IEEEeqnarray}
\end{proof}

\subsection{Bound on error term due to MLE}
\label{sec:boundErrMLE}
We bound the third term in Equation~\ref{eqn:boundEMDFingerprintEst} using the following lemma.
\begin{lemma}\label{lemma:mlefingerprint}
For $3 \leq t \leq \sqrt{C_0 N} + 2$, where $C_0 > 0$ is a constant, w. p. $1 - \delta$,
\begin{IEEEeqnarray}{rCl}
\left| \sum_{j=0}^t b_j \left(h_j^{\text{obs}} - \text{E}_{P_\mle}[h_j] \right) \right| 
    &\leq& \max_j |b_j|\ \sum_{j=0}^t \left| \left(h_j^{\text{obs}} - \text{E}_{P_\mle}[h_j] \right) \right|,\nonumber\\
    &\leq& \max_j |b_j|\ \sqrt{ 2\ \text{ln} 2 } \sqrt{\frac{t}{2N} \log{\frac{4N}{t}} + \frac{1}{N} \log{\frac{3e}{\delta}} }.
\end{IEEEeqnarray} 
\end{lemma}
\begin{proof}
Let $P_{\mle}$ be an optimal solution to the MLE and $P^\star$ be the true distribution. By optimality of the MLE solution, we have the following inequality,
\begin{IEEEeqnarray}{rCl}
\label{eqn:mleOpt}
\KL(\h^{\obs}, \tE_{P_\mle}[\h]) \leq \KL(\h^{\obs}, \tE_{P^\star}[\h]).
\end{IEEEeqnarray}
\begin{proposition}[Pinsker's Inequality~\cite{cover2012elements}]
For discrete distributions $P$ and $Q$:
\begin{equation}
    \KL(P, Q) \geq \frac{1}{2\text{ln}2} || P - Q ||_1^2.
\end{equation}
\end{proposition}
Using Pinsker's inequality and the optimality of the MLE solution, we can bound term $(c)$ in Equation~\eqref{eqn:boundEMDFingerprintEst} as follows:
\begin{IEEEeqnarray}{rCl}
\left| \sum_{j=0}^t b_j \left(h_j^{\text{obs}} - \text{E}_{P_\mle}[h_j] \right) \right| &\leq&  \sum_{j=0}^t |b_j|\ \left| h_j^{\text{obs}} - \text{E}_{P_\mle}[h_j] \right|, \nonumber\\
    &\leq& \max_j |b_j|\ \sum_{j=0}^t \left| h_j^{\obs} - \text{E}_{P_\mle}[h_j] \right|,\nonumber\\
    &\leq& \max_j |b_j|\ \sqrt{2\ \text{ln}2\ \KL(\h^{\obs}, E_{P_\mle}[\h])}, \\
    &\leq& \max_j |b_j|\ \sqrt{2\ \text{ln}2\ \KL(\h^{\obs}, E_{P^\star}[\h])}.
\end{IEEEeqnarray} 
Using the recent results on bounds on KL divergence between empirical observations and the true distribution for discrete distributions~\cite{mardia2018concentration},
for $3 \leq t \leq \sqrt{C_0 N} + 2$, w. p. $1 - \delta$,
\begin{IEEEeqnarray*}{rCl}
\left| \sum_{j=0}^t b_j \left(h_j^{\text{obs}} - \text{E}_{P_\mle}[h_j] \right) \right|
    \leq \max_j |b_j|\ \sqrt{ 2\ \text{ln} 2 } \sqrt{\frac{t}{2N} \log{\frac{4N}{t}} + \frac{1}{N} \log{\frac{3e}{\delta}} }.
\end{IEEEeqnarray*} 
\end{proof}

\subsection{Bounding the polynomial approximation error}
In this section we bound $\max_j |b_j|$ as well as term $(a)$ in Equation~\ref{eqn:boundEMDFingerprintEst}, both of which are related to polynomial approximation of Lipschitz-1 funcitons using Bernstein polynomials.
Let $f$ be any Lipschitz-1 function on $[0, 1]$. Let $\hat{f}$ be a polynomial approximation of $f$ using Bernstein polynomials of degree $t$:
\begin{equation}
\label{eqn:BernApproxgen}
\hat{f}(x) = \sum_{j=0}^t b_j \binom{t}{j} x^j (1-x)^{t-j} := \sum_{j=0}^t b_j B_j^t(x),
\end{equation}
where, 
\[B_j^t(x) := \binom{t}{j} x^j (1-x)^{t-j},\]
is $j-$th Bernstein polynomial of degree $t$, for $j = 0, 1, ..., t$.
Our goal is to bound the uniform approximation error, 
\[||f-\hat{f}||_\infty := \underset{x \in [0, 1]}{\max}|f(x) - \hat{f}(x)|,\]
while controlling the magnitude $|b_j|$, of the coefficients. We note that $\max_j |b_j|$ appears in the bounds of the error terms $(b)$ and $(c)$ in Equation~\eqref{eqn:boundEMDIntegralFingerprint}. Therefore, it is important to control it while bounding the polynomial approximation error to obtain tight bounds on the Wasserstein-1 metric in different regimes of $t$ and $N$.

Bernstein~\cite{bernstein1912, lorentz1953bernstein} used $t+1$ uniform samples of the function $f$ on $[0, 1]$, $f(\frac{j}{t}),\ j = 0, 1, ..., t$, as the coefficients in Equation~\ref{eqn:BernApproxgen} to prove Weierstrass Approximation Theorem and showed that the uniform approximation error of such an approximation is $||f-\hat{f}||_\infty \leq \frac{C}{\sqrt{t}}$, where $C$ is a constant. Since we are interested in approximating Lipschitz-1 functions on $[0,1]$, the co-efficients used, $\left|f(\frac{j}{t}) \right|$ are constants, as w.l.o.g, $|f| \leq \frac{1}{2}$. However, this approximation is not sufficient to show the bounds in Theorems~\ref{thm:smalltMLE} and~\ref{thm:mediumtMLE}. Therefore, the question of interest is: \emph{Can we obtain better uniform approximation error using Bernstein polynomials with other bounded coefficients?} The following proposition answers this question.
\begin{proposition}
\label{prop:coeffbound}
Any Lipschitz-1 function on $[0, 1]$ can be approximated using Bernstein polynomials (Equation~\ref{eqn:BernApproxgen}) of degree $t$, with an uniform approximation error of
\begin{enumerate}
    \item $\mathcal{O}(\frac{1}{t})$ with $\max_j |b_j| \leq \sqrt{t} 2^{t}$.
    \item $\mathcal{O}(\frac{1}{k})$ with $\max_j |b_j| \leq \sqrt{k} (t+1) e^{\frac{k^2}{t}}$, for $k < t$.
\end{enumerate}
\end{proposition}
We are now ready to prove our main results in Theorem~\ref{thm:smalltMLE} and Theorem~\ref{thm:mediumtMLE}. We postpone the proof of Proposition~\ref{prop:coeffbound} to the end of this section.
\begin{proof}[Proof of Theorem~\ref{thm:smalltMLE}]
The first approximation result in Proposition~\ref{prop:coeffbound} with Lemma~\ref{lemma:concfingerprints} and Lemma~\ref{lemma:mlefingerprint} gives the following bound on the Wasserstein-1 distance between the MLE estimate and the true distribution,
\[W_1(P^\star, \hP_{\mle}) \leq \mathcal{O}\left( \frac{1}{t} \right) + \mathcal{O}\left( 2^t t \sqrt{\frac{1}{N} \log{\frac{1}{\delta}} }  \right), \]
with probability at least $1 - 2\delta.$ Therefore, for $t = \mathcal{O}(\log{N})$, we obtain the bound in Theorem~\ref{thm:smalltMLE}.
\end{proof}
\begin{proof}[Proof of Theorem~\ref{thm:mediumtMLE}]
With $k = \sqrt{t\ \log{N^c}}$, for appropriate choice of $c > 0$, we obtain a bound of $$\max_j |b_j| \leq t^{1/4}(t+1) \left(\log{N^c}\right)^{1/4}  N^c,$$ on the coefficients with a uniform approximation error $\mathcal{O}\left(\frac{1}{\sqrt{t\ \log{N}}} \right)$. Combining this bound with Lemma~\ref{lemma:concfingerprints} and Lemma~\ref{lemma:mlefingerprint} gives the following bound,
\[W_1(P^\star, \hP_{\mle}) \leq \mathcal{O}\left(\frac{1}{\sqrt{t \log{N}}} \right) + \mathcal{O}\left((t\ \log{N^c})^{\frac{1}{4}}(t+1)N^c  \sqrt{\frac{t}{N}\log{\frac{1}{\delta}}} \right),\]
with probability at least $1 - 2\delta.$ Therefore, for $t$ above $\Omega(\log{N})$, and $t = \mathcal{O}\left({N^{2/9 - \epsilon}} \right)$ for $\epsilon>0$ such that $N^{\epsilon} = \mathcal{O}\left(N^{4c/9} c^{1/9}\left(\log{N}\right)^{1/3} \right)$, we obtain the
result in Theorem~\ref{thm:mediumtMLE}.
\end{proof}

\textbf{Proof sketch of Proposition~\ref{prop:coeffbound}:} 
The key idea of our proof is to approximate $f$ using Chebyshev polynomials of degree $k \leq t$, which are then expressed as linear combinations of Bernstein polynomials of degree $t$ to obtain appropriate bounds on the coefficients $|b_j|$.
Let $\tT_m$ denote Chebyshev polynomial of degree $m$ shifted to $[0, 1]$ which satisfy the following recursive relation:
\begin{equation*}
    \tT_m(x) = (4x - 2)\tT_{m-1} - \tT_{m-2}(x),\ m= 2, 3,....,
\end{equation*}
and $\tT_0(x) = 1$, $\tT_1(x) = 2x - 1$. We use the following lemma regarding Chebyshev polynomial approximation. The proof is available in Appendix~\ref{app:prooflemmachebyapprox}.
\begin{lemma}\label{lemma:cheby-approx}
Given any Lipschitz-1 function $f(x)$ on $[0,1]$, there exists a degree $k$ polynomial in the form of $\hat{f}_k(x)=\sum_{m=0}^k a_m\tT_m(x)$ that approximates $f(x)$ with error $||f-\hat{f}_k||_\infty = O(\frac{1}{k})$, where $\tT_m(x)$ denotes Chebyshev polynomial of degree $m$ shifted to $[0, 1]$. Further, the coefficients $(a_0, a_1, a_2,\ldots,a_k)$ satisfies $\|a\|_2\le 1$.
\end{lemma}
Chebyshev polynomial $\tT_m$, can be written in terms of Bernstein-Bezier polynomials of degree $m$ as follows~\cite{Rababah2003}:
\begin{equation}
    \tT_{m}(x) = \sum_{i=0}^m \left(-1\right)^{m-i} \frac{\binom{2m}{2i}}{\binom{m}{i}} B_{i}^m(x).
\end{equation}
Note that the coefficients of $B_{i}^m$ can be at most $2^m$. To prove the first part of the proposition, we approximate the Lipschitz-$1$ function $f(x)$ using degree $k=t$ Chebyshev polynomial approximation provided by Lemma~\ref{lemma:cheby-approx} with error $\mathcal{O}(\frac{1}{t})$. This gives an upper bound of $2^t$ on the coefficients of Bernstein polynomial. 
To show the second part of the proposition, we need to bound the coefficients of the Bernstein polynomials when the degree of Chebyshev approximation satisfies $k < t$.\\
{\textbf{Degree raising:}} Bernstein polynomials of degree $m < t$ can be raised to degree $t$ as:
\begin{equation}
    B_{i}^m(x) = \sum_{j=i}^{i + t - m} \frac{\binom{m}{i} \binom{t-m}{j-i}}{\binom{t}{j}} B_{j}^t(x).
\end{equation}
Using degree raising of Bernstein polynomials, we can write shifted Chebyshev polynomials of degree $m < t$ in terms of Bernstein polynomials of degree $t$ as,
\begin{IEEEeqnarray}{rCl}
 \tT_{m}(x) &=& \sum_{i=0}^m \left(-1\right)^{m-i} \frac{\binom{2m}{2i}}{\binom{m}{i}} \sum_{j=i}^{i + t - m} \frac{\binom{m}{i} \binom{t-m}{j-i}}{\binom{t}{j}} B_{j}^t(x), \nonumber\\
 &=:& \sum_{j=0}^t C(t, m, j) B_{j}^t(x),
\end{IEEEeqnarray}
where the coefficient of $j$-th Bernstien polynomial of degree $t$ is given by\footnote{For positive integers $a , b > 0$, $\binom{a}{b} = 0$ when $a < b$.},
\begin{IEEEeqnarray}{rCl}
C(t, m, j) : = \sum_{l=0}^{j} \left(-1\right)^{m-l} \frac{\binom{2m}{2l} \binom{t-m}{j-l}}{\binom{t}{j}}.
\end{IEEEeqnarray}
With the closed form expressions of the coefficients of the Bernstein polynomials, what remains is to establish upper bounds of these coefficients. The following is a generating function for $C(t, m,j) \binom{t}{j}$,
\begin{IEEEeqnarray}{rCl}
(1 + z)^{t-m} \frac{(1 + i\ \sqrt{z})^{2m} + (1 - i\ \sqrt{z})^{2m} }{2} 
=\sum_{j = 0}^t  C(t, m,j) \binom{t}{j}  z^{t-j}.
\end{IEEEeqnarray}
Using Beta function, the binomial terms in the denominator can be written as, $$\binom{t}{j}^{-1} = (t+1)\int_0^1 (1-u)^ju^{t-j}du.$$ 
Combining, we obtain the following generating function for the coefficients,
\begin{IEEEeqnarray*}{rCl}
\frac{\sum_{j = 0}^t  C(t,m,j) z^{t-j} }{t+1} =\int_0^1(1-u)+uz)^{t-m}\frac{(\sqrt{1-u}+i\sqrt{uz})^{2m}+(\sqrt{1-u}-i\sqrt{uz})^{2m}}{2} du
\end{IEEEeqnarray*}
We bound the generating function of the coefficients on the unit circle and use Parseval's theorem to prove the following lemma (details are available in Appendix~\ref{app:boundgenfun}).
\begin{lemma}\label{lem:l2normbound}
The $l_2$-norm of the coefficients of $B_j^t$ can be bounded as follows,
\begin{equation}
\label{eqn:l2normbound}
    \sqrt{\sum_{j=0}^t |C(t, m, j)|^2} \leq (t+1) e^{\frac{m^2}{t}}.
\end{equation}
And, hence the coefficients of $B_j^t$ can be bounded as follows,
\begin{IEEEeqnarray}{rCl}
\label{eqn:coeffbound}
|C(t, m, j)| \leq (t+1) e^{\frac{m^2}{t}}.
\end{IEEEeqnarray}
\end{lemma}

Let $f$ be a Lipschitz-1 function on $[0, 1]$. Let $f_k$ be the polynomial approximation using Chebyshev polynomials upto degree $k$ obtained from Lemma~\ref{lemma:cheby-approx}. We re-write each $\tT_m$ using Bernstein polynomials of degree $k$ followed by degree raising to $t$.
\begin{IEEEeqnarray}{rCl}
 f_k(x) &=& \sum_{m=0}^k a_m \tT_m(x) = \sum_{m=0}^k a_m \left( \sum_{j=0}^t C(t, m, j) B_{j}^t(x) \right), \nonumber\\
    &=& \sum_{j = 0}^t \left( \sum_{m=0}^{k} a_m\ C(t, m, j)\right) B_{j}^t(x), \nonumber\\
    &=:& \sum_{j=0}^t b_j B_{j}^t(x).
\end{IEEEeqnarray}
Since $||a||_2 \leq 1$, and from Equation~\ref{eqn:coeffbound}, we obtain the following bound on the coefficients, for $j = 0, 1, 2, ...., t$,
\begin{IEEEeqnarray}{rCl}
\label{eqn:boundCoeffOverall}
 |b_j| &=& \left|\sum_{m=0}^{k} a_m\ C(t, m, j) \right| \leq  \sum_{m=0}^{k} |a_m|\ \left| C(t, m, j) \right|,\nonumber\\
 &\leq& \sqrt{k}\ \max_{m} |C(t, m, j)| \label{eqn:boundonbj}\leq \sqrt{k} (t+1)e^{\frac{k^2}{t}}.
\end{IEEEeqnarray}
Lemma~\ref{lemma:cheby-approx} and Lemma~\ref{lem:l2normbound} together prove Proposition~\ref{prop:coeffbound}.

\subsection{Lower bound for medium $t$ regime}
The basic idea of the proof of Theorem~\ref{thm:lowerboundtlogN} is to construct a pair of distributions $P, Q$ whose first $\Theta(\log N)$ moments match and $W_1(P,Q)=\Theta(\frac{1}{\sqrt{t\log N}})$. With $N$ coins sampled from these distributions, each with $t$ flips, we argue that it is information theoretically hard to distinguish the two distributions. We need the following two propositions for the proof, where Proposition~\ref{prop:prior-exist} gives the existence of such a pair of distributions and Proposition~\ref{prop:tv-single} shows they are hard to distinguish. The proofs of these propositions are provided in Appendix~\ref{app:lowerboundproof}.
\begin{proposition}\label{prop:prior-exist}
For any positive integer $s$, there exists a pair of distributions $P$ and $Q$ supported on $[a,b]$ where $0<a<b$ such that $P$ and $Q$ have identical first $s$ moments, and $W_1(P, Q)\ge \frac{(b-a)}{2s}$.
\end{proposition}
\begin{proposition}\label{prop:tv-single}
Let $P$ and $Q$ be two distributions, supported on $\left[\frac{1}{2}-\sqrt{\frac{\log N}{t}},\   \frac{1}{2}+\sqrt{\frac{\log N}{t}}\right]$, whose first $L:=e^4\log N$ moments match. Let $p\sim P$ and $q\sim Q$. Let $X\sim \text{Binomial}(t,p)$ and $Y\sim \text{Binomial}(t,q)$. Then the total variation distance between $X$ and $Y$ satisfies,
\[
    \text{TV}(X,Y) \le \frac{2\sqrt{t}}{N^{e^4}}.
\]
\end{proposition}
The proof of Theorem~\ref{thm:lowerboundtlogN} follows from these two propositions.
\begin{proof}[Proof of Theorem~\ref{thm:lowerboundtlogN}]
We first apply Proposition~\ref{prop:prior-exist} to construct a pair of distributions $P$ and $Q$ supported on $\left[\frac{1}{2}-\sqrt{\frac{\log N}{t}}, \frac{1}{2}+\sqrt{\frac{\log N}{t}}\right]$ such that their first $L:=e^4\log N$ moments match, and $W_1(P, Q)\ge \frac{1}{e^4}\frac{1}{\sqrt{t\log N}}$. Let $\mathbf{X} := \{X_i\}_{i=1}^N$ be random variables with $X_i \sim \text{Binomial}(t, p_i)$ where $p_i$ is drawn independently from $P$. Let $\mathbf{Y} := \{Y_i\}_{i=1}^N$ be random variables with $Y_i \sim \text{Binomial}(t, q_i)$ where $q_i$ is drawn independently from $Q$. Denote $P_N$ as the joint distribution of $\mathbf{X}$ and $Q_N$ as the joint distribution of $\mathbf{Y}$. It follows from Proposition~\ref{prop:tv-single} that $\text{TV}(X_i,Y_i)\le \frac{2\sqrt{t}}{N^{e^4}}$. By the property of the product distribution, for $t\le \frac{N^{2(e^4-1)}}{36}$, $\text{TV}(P_N, Q_N) \le \frac{2\sqrt{t}}{N^{e^4-1}}\le \frac{1}{3}$, which implies that the minimax error is at least $\frac{1}{3e^4\sqrt{t\log N}}$.
\end{proof}

\section{Discussion and Future Directions}
\label{sec:conclusion}
We consider the problem of learning the distribution of parameters over a heterogeneous population and show that the MLE achieves optimal error bounds with respect to Wasserstein-1 distance in the sparse observation setting. A future direction of work is to incorporate prior knowledge about the properties of the underlying distribution, such as, smoothness, as additional constraints to the MLE optimization problem. Another direction of interest is to extend the analysis to provide guarantees in general Wasserstein-$p$ norms. In a different direction, a natural question of interest is estimating the properties of the underlying distribution. While the MLE can be used as a plug-in solution to estimate properties of the underlying distribution, it is possible that certain properties could be directly estimated more accurately.  Still, one could imagine an analog of the results of~\cite{acharya2017unified}, who showed that in the related setting of drawing samples from a single discrete distribution over a large alphabet, for a large class of properties, applying the plug-in estimator to the results of the ``profile'' maximum likelihood distribution yields a nearly optimal estimator. 

\subsection{Discussion}
\label{sec:coeffboundconj}
We conjecture that the right bound on the coefficients of $B_j^{t}$ for every $m \leq t$ to be,
\begin{equation}
\label{eqn:coeffboundConjecture}
   |C(t, m, j)| \leq e^{\frac{m^2}{t}},\ j = 0, 1, 2, ..., t.
\end{equation}
In fact, for a fixed $m$, the coefficients $C(t, m, j)$ should  converge to points sampled uniformly from $T_m(x)$ as $t \rightarrow \infty$ by Bernstein's approximation. 
So, the bound on the coefficients should converge to $1$ as $t \rightarrow \infty$. 

Furthermore, we believe that the bound on the error due to the MLE could be improved to mirror that for the observed fingerprints, that is, 
$
\left| \sum_{j=0}^t b_j \left(h_j - \text{E}_{P_\mle}(h_j) \right) \right|
    \leq \mathcal{O}\left( \max_j |b_j| \sqrt{\frac{\log{1/\delta}}{N}}\right).
$
Therefore, our conjecture for the upper bound on the range of $t$ in Theorem~\ref{thm:mediumtMLE} is $\mathcal{O}\left({N^{2/3 - \epsilon}} \right)$.

The question of polynomial approximation of Lipschitz-1 functions using Bernstein polynomials with bounded coefficients is an interesting problem on it's own, with implications to general polynomial approximation theory and applications in computer graphics. 
\section*{Acknowledgements}
Sham Kakade acknowledges funding from the Washington Research Foundation for Innovation in Data-intensive Discovery, the National Science Foundation Grant under award CCF-1637360 (Algorithms in the Field) and award CCF-1703574, and the Office of Naval Research (Minerva Initiative) under award N00014-17-1-2313.  Gregory Valiant and Weihao Kong were supported by National Science Foundation award CCF-1704417 and Office of Naval Research award N00014-18-1-2295.

\bibliographystyle{alpha}
\bibliography{refs}

\newcommand{\etalchar}[1]{$^{#1}$}
\begin{thebibliography}{ADM{\etalchar{+}}10}

\bibitem[ADM{\etalchar{+}}10]{acharya2010exact}
Jayadev Acharya, Hirakendu Das, Hosein Mohimani, Alon Orlitsky, and Shengjun
  Pan.
\newblock Exact calculation of pattern probabilities.
\newblock In {\em Information Theory Proceedings (ISIT), 2010 IEEE
  International Symposium on}, pages 1498--1502. IEEE, 2010.

\bibitem[ADOS17]{acharya2017unified}
Jayadev Acharya, Hirakendu Das, Alon Orlitsky, and Ananda~Theertha Suresh.
\newblock A unified maximum likelihood approach for estimating symmetric
  properties of discrete distributions.
\newblock In {\em International Conference on Machine Learning}, pages 11--21,
  2017.

\bibitem[AOP09]{acharya2009recent}
Jayadev Acharya, Alon Orlitsky, and Shengjun Pan.
\newblock Recent results on pattern maximum likelihood.
\newblock In {\em Networking and Information Theory, 2009. ITW 2009. IEEE
  Information Theory Workshop on}, pages 251--255. IEEE, 2009.

\bibitem[BD69]{bojanic1969proof}
R~Bojanic and R~DeVore.
\newblock A proof of jackson’s theorem.
\newblock {\em Bulletin of the American Mathematical Society}, 75(2):364--367,
  1969.

\bibitem[Ber12]{bernstein1912}
S.~Bernstein.
\newblock Démonstration du théorème de weierstrass fondée sur le calcul des
  probabilities.
\newblock {\em Comm. Soc. Math. Kharkov}, 13:1--2, 1912.

\bibitem[BLW00]{bell2000environmental}
Graham Bell, Martin~J Lechowicz, and Marcia~J Waterway.
\newblock Environmental heterogeneity and species diversity of forest sedges.
\newblock {\em Journal of Ecology}, 88(1):67--87, 2000.

\bibitem[B{\"o}h89]{bohning1989likelihood}
Dankmar B{\"o}hning.
\newblock Likelihood inference for mixtures: geometrical and other
  constructions of monotone step-length algorithms.
\newblock {\em Biometrika}, 76(2):375--383, 1989.

\bibitem[CC94]{colwell1994estimating}
Robert~K Colwell and Jonathan~A Coddington.
\newblock Estimating terrestrial biodiversity through extrapolation.
\newblock {\em Phil. Trans. R. Soc. Lond. B}, 345(1311):101--118, 1994.

\bibitem[Cre79]{cressie1979quick}
Noel Cressie.
\newblock A quick and easy empirical bayes estimate of true scores.
\newblock {\em Sankhy{\=a}: The Indian Journal of Statistics, Series B}, pages
  101--108, 1979.

\bibitem[CT12]{cover2012elements}
Thomas~M Cover and Joy~A Thomas.
\newblock {\em Elements of information theory}.
\newblock John Wiley \& Sons, 2012.

\bibitem[ET76]{efron1976estimating}
Bradley Efron and Ronald Thisted.
\newblock Estimating the number of unseen species: How many words did
  shakespeare know?
\newblock {\em Biometrika}, 63(3):435--447, 1976.

\bibitem[FCW43]{fisher1943relation}
Ronald~A Fisher, A~Steven Corbet, and Carrington~B Williams.
\newblock The relation between the number of species and the number of
  individuals in a random sample of an animal population.
\newblock {\em The Journal of Animal Ecology}, pages 42--58, 1943.

\bibitem[GT56]{good1956number}
IJ~Good and GH~Toulmin.
\newblock The number of new species, and the increase in population coverage,
  when a sample is increased.
\newblock {\em Biometrika}, 43(1-2):45--63, 1956.

\bibitem[HJW18]{Han2018}
Yanjun Han, Jiantao Jiao, and Tsachy Weissman.
\newblock Local moment matching: A unified methodology for symmetric functional
  estimation and distribution estimation under wasserstein distance.
\newblock {\em arXiv preprint arXiv:1802.08405}, 2018.

\bibitem[Jac21]{jackson1921general}
Dunham Jackson.
\newblock The general theory of approximation by polynomials and trigonometric
  sums.
\newblock {\em Bulletin of the American Mathematical Society},
  27(9-10):415--431, 1921.

\bibitem[JHW18]{jiao2018minimax}
Jiantao Jiao, Yanjun Han, and Tsachy Weissman.
\newblock Minimax estimation of the l1 distance.
\newblock {\em IEEE Transactions on Information Theory}, 2018.

\bibitem[JVHW15]{jiao2015minimax}
Jiantao Jiao, Kartik Venkat, Yanjun Han, and Tsachy Weissman.
\newblock Minimax estimation of functionals of discrete distributions.
\newblock {\em IEEE Transactions on Information Theory}, 61(5):2835--2885,
  2015.

\bibitem[KR58]{Kantorovich1958}
L.~V. Kantorovich and G.~S. Rubinstein.
\newblock On a space of completely additive functions.
\newblock {\em Vestnik Leningrad. Univ, 13(7):52–59}, 1958.

\bibitem[Lai78]{laird1978nonparametric}
Nan Laird.
\newblock Nonparametric maximum likelihood estimation of a mixing distribution.
\newblock {\em Journal of the American Statistical Association},
  73(364):805--811, 1978.

\bibitem[LC75]{lord1975empirical}
Frederic~M Lord and Noel Cressie.
\newblock An empirical bayes procedure for finding an interval estimate.
\newblock {\em Sankhy{\=a}: The Indian Journal of Statistics, Series B}, pages
  1--9, 1975.

\bibitem[Lin83a]{lindsay1983geometry}
Bruce~G Lindsay.
\newblock The geometry of mixture likelihoods: a general theory.
\newblock {\em The annals of statistics}, pages 86--94, 1983.

\bibitem[Lin83b]{lindsay1983geometry2}
Bruce~G Lindsay.
\newblock The geometry of mixture likelihoods, part ii: the exponential family.
\newblock {\em The Annals of Statistics}, 11(3):783--792, 1983.

\bibitem[LK92]{lesperance1992algorithm}
Mary~L Lesperance and John~D Kalbfleisch.
\newblock An algorithm for computing the nonparametric mle of a mixing
  distribution.
\newblock {\em Journal of the American Statistical Association},
  87(417):120--126, 1992.

\bibitem[Lor53]{lorentz1953bernstein}
George~G Lorentz.
\newblock {\em Bernstein polynomials}.
\newblock Toronto: University of Toronto Press, 1953.

\bibitem[Lor65]{lord1965strong}
Frederic~M Lord.
\newblock A strong true-score theory, with applications.
\newblock {\em Psychometrika}, 30(3):239--270, 1965.

\bibitem[Lor69]{lord1969estimating}
Frederic~M Lord.
\newblock Estimating true-score distributions in psychological testing (an
  empirical bayes estimation problem).
\newblock {\em Psychometrika}, 34(3):259--299, 1969.

\bibitem[Mil86]{millar1986distribution}
Wayne~J Millar.
\newblock Distribution of body weight and height: comparison of estimates based
  on self-reported and observed measures.
\newblock {\em Journal of Epidemiology \& Community Health}, 40(4):319--323,
  1986.

\bibitem[MJT{\etalchar{+}}18]{mardia2018concentration}
Jay Mardia, Jiantao Jiao, Ervin T{\'a}nczos, Robert~D Nowak, and Tsachy
  Weissman.
\newblock Concentration inequalities for the empirical distribution.
\newblock {\em arXiv preprint arXiv:1809.06522}, 2018.

\bibitem[OSVZ04]{orlitsky2004modeling}
Alon Orlitsky, Narayana~P Santhanam, Krishnamurthy Viswanathan, and Junan
  Zhang.
\newblock On modeling profiles instead of values.
\newblock In {\em Proceedings of the 20th conference on Uncertainty in
  artificial intelligence}, pages 426--435. AUAI Press, 2004.

\bibitem[OSW16]{orlitsky2016optimal}
Alon Orlitsky, Ananda~Theertha Suresh, and Yihong Wu.
\newblock Optimal prediction of the number of unseen species.
\newblock {\em Proceedings of the National Academy of Sciences},
  113(47):13283--13288, 2016.

\bibitem[Pan03]{paninski2003estimation}
Liam Paninski.
\newblock Estimation of entropy and mutual information.
\newblock {\em Neural computation}, 15(6):1191--1253, 2003.

\bibitem[PD90]{palmer1990small}
Michael~W Palmer and Philip~M Dixon.
\newblock Small-scale environmental heterogeneity and the analysis of species
  distributions along gradients.
\newblock {\em Journal of Vegetation Science}, 1(1):57--65, 1990.

\bibitem[Rab03]{Rababah2003}
Abedallah Rababah.
\newblock Transformation of chebyshev--bernstein polynomial basis.
\newblock {\em Computational Methods in Applied Mathematics Comput. Methods
  Appl. Math.}, 3(4):608--622, 2003.

\bibitem[Sim76]{simar1976maximum}
Leopold Simar.
\newblock Maximum likelihood estimation of a compound poisson process.
\newblock {\em The Annals of Statistics}, pages 1200--1209, 1976.

\bibitem[TKV17]{Tian2017}
Kevin Tian, Weihao Kong, and Gregory Valiant.
\newblock Optimally learning populations of parameters.
\newblock {\em arXiv preprint arXiv:1709.02707}, 2017.

\bibitem[Tur76]{turnbull1976empirical}
Bruce~W Turnbull.
\newblock The empirical distribution function with arbitrarily grouped,
  censored and truncated data.
\newblock {\em Journal of the Royal Statistical Society. Series B
  (Methodological)}, pages 290--295, 1976.

\bibitem[Von12]{vontobel2012bethe}
Pascal~O Vontobel.
\newblock The bethe approximation of the pattern maximum likelihood
  distribution.
\newblock In {\em Information Theory Proceedings (ISIT), 2012 IEEE
  International Symposium on}. IEEE, 2012.

\bibitem[VV11a]{valiant2011estimating}
Gregory Valiant and Paul Valiant.
\newblock Estimating the unseen: an n/log (n)-sample estimator for entropy and
  support size, shown optimal via new clts.
\newblock In {\em Proceedings of the forty-third annual ACM symposium on Theory
  of computing}, pages 685--694. ACM, 2011.

\bibitem[VV11b]{valiant2011power}
Gregory Valiant and Paul Valiant.
\newblock The power of linear estimators.
\newblock In {\em Foundations of Computer Science (FOCS), 2011 IEEE 52nd Annual
  Symposium on}, pages 403--412. IEEE, 2011.

\bibitem[VV13]{valiant2013estimating}
Paul Valiant and Gregory Valiant.
\newblock Estimating the unseen: improved estimators for entropy and other
  properties.
\newblock In {\em Advances in Neural Information Processing Systems}, pages
  2157--2165, 2013.

\bibitem[VV16]{valiant2016instance}
Gregory Valiant and Paul Valiant.
\newblock Instance optimal learning of discrete distributions.
\newblock In {\em Proceedings of the forty-eighth annual ACM symposium on
  Theory of Computing}, pages 142--155. ACM, 2016.

\bibitem[Woo99]{wood1999binomial}
G.~R. Wood.
\newblock Binomial mixtures: geometric estimation of the mixing distribution.
\newblock {\em The Annals of Statistics}, 27(5):1706--1721, 1999.

\bibitem[WY15]{wu2015chebyshev}
Yihong Wu and Pengkun Yang.
\newblock Chebyshev polynomials, moment matching, and optimal estimation of the
  unseen.
\newblock {\em arXiv preprint arXiv:1504.01227}, 2015.

\bibitem[WY16]{wu2016minimax}
Yihong Wu and Pengkun Yang.
\newblock Minimax rates of entropy estimation on large alphabets via best
  polynomial approximation.
\newblock {\em IEEE Transactions on Information Theory}, 62(6):3702--3720,
  2016.

\end{thebibliography}

\appendix
\section{Proofs of Lemma~\ref{lemma:cheby-approx} and Lemma~\ref{lem:l2normbound}}
In this section we provide proofs for  Lemma~\ref{lemma:cheby-approx} and Lemma~\ref{lem:l2normbound} used to prove Theorem~\ref{thm:smalltMLE} and Theorem~\ref{thm:mediumtMLE}.

\subsection{Chebyshev polynomial approximation: Proof of Lemma~\ref{lemma:cheby-approx}}
\label{app:prooflemmachebyapprox}
In this section, we focus on proving Lemma~\ref{lemma:cheby-approx}. The existence of the Chebyshev polynomial approximation with error $\mathcal{O}(1/k)$ is shown using the following result,
\begin{lemma}~\cite{jackson1921general, bojanic1969proof}\label{lemma:exist-chebyshev}
Given any Lipschitz-1 function $f(x)$ on $[0,1]$, there exists a degree $k$ polynomial in the form of $f_k(x)=\sum_{m=0}^k a_m\tT_m(x)$ that approximates $f(x)$ with error $\max_{x\in [0,1]}|f(x)-f_k(x)| = O(\frac{1}{k})$, where $\tT_m(x)$ denotes Chebyshev polynomial of degree $m$ shifted to $[0, 1]$. 
\end{lemma}

We now show that the coefficients satisfy $\|a\|\_2^2le 1$. 
Let $\tT_m$ denote Chebyshev polynomial of degree $m$ shifted to $[0, 1]$ which satisfy the following recursive relation: 
\[\tT_m(x) = (4x - 2)\tT_{m-1} - \tT_{m-2}(x),\ m= 2, 3,....,\]
and $\tT_0(x) = 1$, $\tT_1(x) = 2x - 1$.
Shifted Chebyshev polynomials form a sequence of orthogonal polynomials with respect to the weight $\frac{1}{\sqrt{4x - 4x^2}}$:
\begin{equation}
    \int_{0}^1 \tT_m(x) \tT_n(x) \frac{dx}{\sqrt{4x - 4x^2}} = \left\{
	\begin{array}{ll}
		0  & \mbox{if } m \neq n \\
		\frac{\pi}{2} & \mbox{if } m = n = 0\\
		\frac{\pi}{4} & \mbox{if } m = n \neq 0.
	\end{array}
\right.
\end{equation}
Let $f$ be a Lipschitz-1 function on $[0, 1]$. Let $f_k$ be degree $k$ polynomial approximation of $f$ using Chebyshev polynomials up to degree $k$,
\begin{IEEEeqnarray}{rCl}
\label{eqn:chebyshevApprox}
    f_k(x) &=& \sum_{m=0}^k a_m \tT_m(x).
\end{IEEEeqnarray}
Since $f$ is Lipschitz-1 on $[0, 1]$, w.l.o.g. $|f_k(x)| \leq 1/2$ for all $x \in [0, 1]$. So, $|a_m| \leq 1$.
Furthermore, the norm of the coefficient vector can be bounded as follows:
\begin{IEEEeqnarray}{rCl}
\int_0^1 |f_k(x)|^2 \frac{dy}{\sqrt{4x - 4x^2}}
&=& \sum_{m, n=0}^k a_m a_n \int_0^1   \tT_m(x) \tT_n(x) \frac{dy}{\sqrt{4x - 4x^2}}\nonumber\\
&=&a_0^2\ \frac{\pi}{2} + \sum_{m = 1}^k a_m^2\ \frac{\pi}{4} = a_0^2\ \frac{\pi}{4} + \sum_{m = 0}^k a_m^2\ \frac{\pi}{4}.
\end{IEEEeqnarray}
Since $|f_k(x)| \leq 1/2$ for all $x \in [0, 1]$ and $\int_0^1 \frac{dy}{\sqrt{4x - 4x^2}} = \frac{\pi}{2}$, we obtain the following bound,
\begin{equation}
    a_0^2\ \frac{\pi}{4} + ||a||_2^2 \frac{\pi}{4} \leq \frac{\pi}{8}.
\end{equation}
Hence, $||a||_2^2 \leq 1.$ Along with Lemma~\ref{lemma:exist-chebyshev}, this completes the proof of Lemma~\ref{lemma:cheby-approx}.

\subsection{Bound on generating function: Proof of Lemma~\ref{lem:l2normbound}}
\label{app:boundgenfun}
Using degree raising of Bernstein polynomials, we can write the shifted Chebyshev polynomials of degree $m < t$ in terms of Bernstein polynomials of degree $t$ as follows,
\begin{IEEEeqnarray}{rCl}
 \tT_{m}(x) &=& \sum_{i=0}^m \left(-1\right)^{m-i} \frac{\binom{2m}{2i}}{\binom{m}{i}} \sum_{j=i}^{i + t - m} \frac{\binom{m}{i} \binom{t-m}{j-i}}{\binom{t}{j}} B_{j}^t(x), \nonumber\\
 &=:& \sum_{j=0}^t C(t, m, j) B_{j}^t(x),
\end{IEEEeqnarray}
where the coefficient of $j$-th Bernstein polynomial of degree $t$ is given by\footnote{For positive integers $a , b > 0$, $\binom{a}{b} = 0$ when $a < b$.},
\begin{IEEEeqnarray}{rCl}
C(t, m, j) : = \sum_{l=0}^{j} \left(-1\right)^{m-l} \frac{\binom{2m}{2l} \binom{t-m}{j-l}}{\binom{t}{j}}.
\end{IEEEeqnarray}
Following is a generating function for the coefficients multiplied by the Binomial terms,
\begin{IEEEeqnarray}{rCl}\label{eqn:genfunCoeffBinom}
(1 + z)^{t-m} \frac{(1 + i\ \sqrt{z})^{2m} + (1 - i\ \sqrt{z})^{2m} }{2} \nonumber\\
=\sum_{j = 0}^t  C(t, m,j) \binom{t}{j}  z^{t-j}.
\end{IEEEeqnarray}
Using the Beta function the binomial terms in the denominator can be written as, 
$$\binom{t}{j}^{-1} = (t+1)\int_0^1 (1-u)^ju^{t-j}du.$$
Combing with Equation~\ref{eqn:genfunCoeffBinom}, we obtain the following generating function for the coefficients,

\begin{IEEEeqnarray}{rCl}
    &&(t+1)\int_0^1((1-u)+uz)^{t-m}\frac{(\sqrt{1-u}+i\sqrt{uz})^{2m}+(\sqrt{1-u}-i\sqrt{uz})^{2m}}{2} du\label{eqn:gen-full}\\
    &&=\sum_{j = 0}^t  \left( \frac{\sum_{l = 0}^j (-1)^{m-l} \binom{2m}{2l}  \binom{t-m}{j - l}}{\binom{t}{j}} \right) z^{t-j} \label{eqn:rhs-gen-fun}
\end{IEEEeqnarray}

Our goal is to bound the generating function Equation~\ref{eqn:gen-full} on the unit circle.
Let $i\sqrt{z} = \cos\theta +i\sin\theta$, then $z = \cos(\pi+2\theta)+i \sin(\pi+2\theta) = -\cos(2\theta)-i \sin(2\theta)$. We bound the norm of  $((1-u)+uz)^{t-m}\frac{(\sqrt{1-u}+i\sqrt{uz})^{2m}+(\sqrt{1-u}-i\sqrt{uz})^{2m}}{2}$ by bounding the norm of $(1-u)+uz$ and $(\sqrt{1-u}+i\sqrt{uz})$ which can be expresses as follows,
\begin{IEEEeqnarray*}{rCl}
|(1-u)+uz|^2 &=& (1-u(1+\cos(2\theta)))^2+u^2sin^2(2\theta) \\
&=& 1 + (2u^2 - 2u) (1 + cos 2\theta);
\end{IEEEeqnarray*}
\begin{IEEEeqnarray*}{rCl}
|\sqrt{1-u}+i\sqrt{uz}|^2 = (\sqrt{1-u}+\sqrt{u}\cos\theta)^2+u\sin^2\theta = 1 + 2 \sqrt{u - u^2} cos \theta.
\end{IEEEeqnarray*}
Thus, we can bound the logarithm of the magnitude of the integrand as follows,
\begin{IEEEeqnarray}{rCl}
&&\log\left|((1-u)+uz)^{t-m}\left(\frac{(\sqrt{1-u}+i\sqrt{uz})^{2m}+(\sqrt{1-u}-i\sqrt{uz})^{2m}}{2}\right)\right|\nonumber\\
&&\le \frac{t-m}{2} \log((1-u(1+\cos(2\theta)))^2+u^2sin^2(2\theta))+m\log((\sqrt{1-u}+\sqrt{u}\cos\theta)^2+u\sin^2\theta)\nonumber\\
&&=: B(\theta, v).\label{eqn:boundonLogMagGenfun}
\end{IEEEeqnarray}
Let $v = \frac{1}{2}-u$. Taking the derivative of $B(\theta, v)$ with respect to $\theta$ and $v$, we get the following two expressions:
\begin{align*}
\frac{dB(\theta, v)}{d\theta} &= -4 v \cos (\theta ) \left(\frac{2 (m-t) \cos (\theta )}{\left(4 v^2-1\right) \cos (2
   \theta )+4 v^2+1}+\frac{m}{-4 v^2 \cos (\theta )+\cos (\theta )+\sqrt{1-4
   v^2}}\right)\\
\frac{dB(\theta, v)}{dv} &= \frac{\left(4 v^2-1\right) (m-t) \sin (2 \theta )}{\left(4 v^2-1\right) \cos (2 \theta
   )+4 v^2+1}-\frac{m \sqrt{1-4 v^2} \sin (\theta )}{\sqrt{1-4 v^2} \cos (\theta )+1}
\end{align*}
In order to find the maximum of the function, we solve for $\theta, v$ such that the above two expressions equal $0$. Ignoring the solutions where $\theta = 0$ which are clearly not the maximum, we have $\theta = \arccos[m/t]$ and $v=0$. 
Plugging in the solution to the upper bound in Equation~\ref{eqn:boundonLogMagGenfun}, we obtain the following upper bound on the logarithm of the magnitude of the integrand of the generating function,
$$
B(\theta,v)\le \underset{<0}{\underbrace{\frac{t-m}{2}\log{\left(1-\frac{m^2}{t^2}\right)}}}+m\log
{\left(1+\frac{m}{t}\right)} \le m\log{\left(1+\frac{m}{t}\right)} \le \frac{m^2}{t}.
$$
Therefore, Equation~\ref{eqn:gen-full} can be bounded by $e^{m^2/t}(t+1)$. Hence, for all $z$ on the unit circle, Equation~\ref{eqn:rhs-gen-fun} is bounded by $(t+1)e^{m^2/t}$. 
\begin{proposition}\label{prop:gen-bound-coeff}
Given a degree $d-1$ real polynomial $p(x) = \sum_{i=0}^{d-1}a_ix^i$ that satisfies $|p(x)|\le c$ for all complex numbers $|x|=1$, the sum of the squares of the coefficients satisfies $\sum_{i=0}^{d-1}a_i^2\le c^2$. 
\end{proposition}
\begin{proof}
Let $g_k = p(x_k)$ where  $x_k = e^{-\frac{2\pi i}{d}k}$ for $k=0,1,\ldots,d-1$. By the assumption, we have $|g_k|\le c$ for all $k$. Notice that $(g_0, g_1, \ldots, g_{d-1})$ is the discrete Fourier transform of the coefficient vector $(a_0,a_1,\ldots,a_{d-1})$. Hence, by Parseval's theorem, we have $\sum_{i=0}^{d-1} a_i^2 = \frac{1}{d}\sum_{k=0}^{d-1} g_k^2 \le c^2$.
\end{proof}
From Proposition~\ref{prop:gen-bound-coeff}, we obtain the following bound the $l_2$-norm of the coefficients,
\begin{equation}
\label{eqn:l2normbound2}
    \sqrt{\sum_{j=0}^t |C(t, m, j)|^2} \leq (t+1) e^{\frac{m^2}{t}}.
\end{equation}
Using the above bound, each of the coefficient can be bounded as follows:
\begin{equation*}
    |C(t, m, j)| \leq (t+1) e^{\frac{m^2}{t}}.
\end{equation*}
This completes the proof of Lemma~\ref{lem:l2normbound}.

\section{Proof of Theorem~\ref{thm:lowerboundtlogN}}
\label{app:lowerboundproof}
In this section, we provide the detailed proofs of Proposition~\ref{prop:prior-exist} and Proposition~\ref{prop:tv-single} that are used to prove Theorem~\ref{thm:lowerboundtlogN}.

\subsection{Proof of Proposition~\ref{prop:prior-exist}}
Proposition~\ref{prop:prior-exist} states the following: 
For any positive integer $s$, there exists a pair of distributions $P, Q$ supported on $[a,b]$ where $0<a<b$ such that $P$ and $Q$ have identical first $s$ moments, and $W_1(P, Q)\ge \frac{(b-a)}{2s}$
\begin{proof}
Our proof leverages the following result from ~\cite{Tian2017},

\begin{lemma}[\cite{Tian2017}, Lemma 3]\label{lemma:mmm}
For any positive integer $s$, there exists a pair of distributions $P', Q'$ supported on $[0,1]$ such that $P'$ and $Q'$ have identical first $s$ moments, and $W_1(P', Q')\ge \frac{1}{2s}$.
\end{lemma}

The pair of distributions $P'$ and $Q'$ supported on $[0,1]$ can be transformed to a pair of distributions $P$ and $Q$ supported on $[a,b]$, where $0<a<b$, via transformation $P(x) = \frac{1}{b-a}P'\left(\frac{x-a}{b-a}\right)$. We show that $P$ and $Q$ have identical first $s$ moments as follows. For $k \leq s$,
\begin{IEEEeqnarray*}{rCl}
\int_a^b P(x)\ x^k dx &=& \int_a^b \frac{1}{b-a}P'\left(\frac{x-a}{b-a}\right)x^k dx\\
&=& \int_0^1 P'(y)\ (y(b-a)+a)^k dy\\
&=& \int_0^1 Q'(y)\ (y(b-a)+a)^k dy\\
&=& \int_a^b Q(x)\ x^k dx.
\end{IEEEeqnarray*}

\end{proof}

\subsection{Proof of Proposition~\ref{prop:tv-single}}
Proposition~\ref{prop:tv-single} states the following:
Let $P$ and $Q$ be two distributions supported on the interval $\left[\frac{1}{2}-\sqrt{\frac{\log N}{t}},\ \frac{1}{2}+\sqrt{\frac{\log N}{t}}\right]$, whose first $L:=e^4\log N$ moments match. Let $p\sim P$, $X\sim Binomial(t,p)$, $q\sim Q$ and $Y\sim Binomial(t,q)$. The total variation distance between $X$ and $Y$ satisfies
$$
\text{TV}(X,Y) \le \frac{2\sqrt{t}}{N^{e^4}}.
$$
\begin{proof}
The total variation distance between $X$ and $Y$ is 
\begin{IEEEeqnarray*}{rCl}
\text{TV}(X,Y) 
&=& \sum_{j = 0}^t \left|\mE_{p\sim P}[\binom{t}{j}p^j(1-p)^{t-j}]-\mE_{q\sim Q}[\binom{t}{j}q^j(1-q)^{t-j}]\right|\\
&=& \sum_{j = 0}^t \left|\mE_{p\sim P}[B_j(p)]-\mE_{q\sim Q}[B_j(q)]\right|,
\end{IEEEeqnarray*}
where $B_j(p) = \binom{t}{j}p^j(1-p)^{t-j}$ is the $j-$th Bernstein polynomial of degree $t$. Expanding the Bernstein polynomial at $p=\frac{1}{2}$ we get,
\begin{align*}
B_j(p) &= \sum_{k=0}^\infty \frac{B_j^{ (k)}\left(\frac{1}{2}\right)}{k!}\left(p-\frac{1}{2}\right)^k,
\end{align*}
where $B_j^{ (k)}(a) := \left.\frac{d^k B_j(x)}{dx^k}\right|_{x = a}$. 
Therefore, we can bound the total variation distance between $X$ and $Y$ as follows,
\begin{IEEEeqnarray*}{rCl}
\text{TV}(X,Y)
&=& \sum_{j = 0}^t \left|\sum_{k=0}^\infty \frac{B_j^{ (k)}\left(\frac{1}{2}\right)}{k!}\ \mE\left[\left(p- \frac{1}{2}\right)^k\right]-\sum_{k=0}^\infty \frac{B_j^{(k)}\left(\frac{1}{2}\right)}{k!}\mE\left[\left(q-\frac{1}{2}\right)^k\right] \right|\\
&&\le \sum_{k=0}^\infty \sum_{j = 0}^t \left|\frac{B_j^{ (k)}\left(\frac{1}{2}\right)}{k!}\right| \left|\mE\left[\left(p - \frac{1}{2}\right)^k\right]-\mE\left[\left(q- \frac{1}{2}\right)^k\right]\right|\\
&&\le 2\sum_{k=L+1}^\infty \sum_{j = 0}^t \left|\frac{B_j^{ (k)}\left(\frac{1}{2}\right)}{k!}\right|\left(\frac{\log N}{t}
\right)^{\frac{k}{2}}.
\end{IEEEeqnarray*}
The last inequality follows from the fact that $\mE\left[\left(p- \frac{1}{2}\right)^k\right]\le \left(\frac{\log N}{t}
\right)^{\frac{k}{2}}$, since $P$ and $Q$ are supported on $\left[\frac{1}{2}-\sqrt{\frac{\log N}{t}},\frac{1}{2}+\sqrt{\frac{\log N}{t}}\right]$ and their first $L$ moments match. Further, applying Proposition~\ref{prop:Bern-half-bound}, we obtain the following bound, 
\begin{align*}
\text{TV}(X,Y)&\le 2\sum_{k=L+1}^\infty 
\frac{\sqrt{t}e^{k}t^{k/2}}{k^{k/2}}
\left(\frac{\log N}{t}
\right)^{\frac{k}{2}} = 2\sum_{k=L+1}^\infty  \frac{\sqrt{t}e^k(\log N)^{k/2}}{k^{k/2}} 
,\\
&\le 2\sqrt{t} \sum_{k=L+1}^\infty \left(\frac{e^2\log N}{L}\right)^{k/2},\\
& = 2\sqrt{t} \sum_{k=L+1}^\infty \left(\frac{1}{e}\right)^{k},\ (\text{since } L = e^4 \log{N}),\\
&\le 2 \sqrt{t} e^{-e^4\log N} \le
\frac{2\sqrt{t}}{N^{e^4}}.
\end{align*}
\end{proof}
\begin{proposition}\label{prop:Bern-half-bound}
$$
\sum_{j=0}^t \left|\frac{B_j^{(k)}\left(\frac{1}{2}\right)}{k!} \right| \le  \frac{\sqrt{t}e^{k}t^{k/2}}{k^{k/2}}.
$$

\end{proposition}
\begin{proof}
Each term involving the $k$-th order derivative of Bernstein polynomial $B_j(x)$ evaluated at $\frac{1}{2}$ can be written as:
\begin{align*}
\frac{B_j^{ (k)}\left(\frac{1}{2}\right)}{k!}
&= \frac{1}{2^t}\binom{t}{k}\sum_{i=0}^k\binom{k}{i}\binom{t-k}{j-i}(-1)^{(k-i)},\\
&= (-1)^k\frac{1}{2^t}\binom{t}{k}K_j(k;t),
\end{align*}
where $K_j(k;t)$ is the Kravchuk polynomial. 

We then apply Proposition~\ref{prop:gen-bound-coeff} to obtain an upper bound for $\sqrt{\sum_{j=0}^tK_j^2(k;t)}$ by bounding the generating function of Kravchuk Polynomial. The generating function of the Karvchuk polynomials is $$\sum_{j=0}^tK_j(k; t)\ z^j = (1+z)^{t-k}(1-z)^k.$$ 
Our aim is to bound the absolute value of the generating function evaluated on the unit circle. Define $z = \cos\theta+i\sin\theta$. The absolute value of the generating function is  $|(1+z)^{t-k}(1-z)^k| = 2^t \cos^{t-k}(\frac{\theta}{2})\sin^k(\frac{\theta}{2})$ which achieves maximum at $\theta =2\arcsin(\sqrt{k/t})$ with value $2^t(\frac{t-k}{t})^{(t-k)/2}(\frac{k}{t})^{k/2}$. 
Hence, it follows from Proposition~\ref{prop:gen-bound-coeff} that
\begin{IEEEeqnarray*}{rCl}
\sum_{j=0}^t \left|\frac{B_j^{(k)}\left(\frac{1}{2}\right)}{k!}\right| &\le& \frac{1}{2^t}\binom{t}{k}\sqrt{t}\sqrt{\sum_{j=0}^t K_j^2(k;t)} \\
&\le& \binom{t}{k}\sqrt{t}\left(\frac{t-k}{t}\right)^{(t-k)/2}\left(\frac{k}{t}\right)^{k/2}.
\end{IEEEeqnarray*}
Finally, it follows from $\binom{t}{k}\le (\frac{e t}{k})^k$ and $(1-\frac{k}{t})^{(t-k)/2} \le {e^{(k^2/t-k)/2}}$ that
$$
\sum_{j=0}^t\left|\frac{B_j^{(k)}\left(\frac{1}{2}\right)}{k!}\right| \le \frac{\sqrt{t}e^{(k^2/(t-k))/2}t^{k/2}}{k^{k/2}}\le \frac{\sqrt{t}e^{k}t^{k/2}}{k^{k/2}}.
$$
\end{proof}

\section{Local moment matching}
\label{app:localmomentmatching}
In this section, we provide a high-level idea of how the local moment matching~\cite{jiao2018minimax} can be extended to obtain an algorithm to estimate the distribution of the parameters. The algorithm consists of three steps:
\begin{enumerate}
    \item \textbf{Binning:}
We divide the coin flips into two batches. The first batch of the data consists of the result of the first $t/2$ coin flips of each coin, and we call the first batch of samples $X'_1, X'_2,\ldots, X'_N$ and the second batch of samples $X_1, X_2,\ldots, X_N$. We define disjoint intervals \begin{equation}
\label{eqn:intervals}
I_j := \left[ \frac{(j-1)^2\ c_1 \log{N}}{t},  \frac{j^2\ c_1 \log{N}}{t}\right],
\end{equation} 
for $j = 1, 2, ... M := \sqrt{\frac{t}{c_2 \log{N}}}$ (assuming $\sqrt{\frac{t}{c_2 \log{N}}}$ be an integer). We define $l_j = \frac{(j-1)^2\ c_1 \log{N}}{t}, r_j = \frac{j^2\ c_1 \log{N}}{t}$ to be the left and right end of the $i$-th interval. The $i$-th coin is assigned to interval (bin) $j$ if $X'_i/t\in I_j$.

\item \textbf{Moment estimation:}
In the second step, we estimate the first $c_2\log N$'th moments of the coins in each interval (bin). Here the $k$-th moment of the coins in the $j$-th bin is defined to be $m_k^{(j)}:=\sum_{X'_i/t\in I_j} (p_i-l_j)^k$. It follows from Lemma 1 of~\cite{Tian2017} that $\frac{\binom{X_i}{l}}{\binom{t}{l}}$ is an unbiased estimator for $p_i^l$. Hence, $\sum_{l=0}^k\binom{k}{l}(-l_j)^{k-l}\frac{\binom{X_i}{l}}{\binom{t}{l}}$ is an unbiased estimator of $(p_i-l_j)^k$. We compute $\sum_{X'_i/t\in I_j}\sum_{l=0}^k\binom{k}{l}(-l_j)^{k-l}\frac{\binom{X_i}{l}}{\binom{t}{l}}$ as an estimate of the $k$-th moment of the $j$-th bin for all $k=0,1\dots,c_2\log N$, and denote it as $\hat{m}_k^{(j)}$.

\item \textbf{Distribution recovery:}
In the third step, for each bin $j$, we solve a linear programming (see e.g. Algorithm 1 of~\cite{Tian2017}) to recover a distribution $\mu_j$ supported on $\tilde{I}_j := [\frac{(j-\frac{3}{2})^2\ c_1 \log{N}}{t},  \frac{(j+1)^2\ c_1 \log{N}}{t}]$ whose first $c_2\log N$ moments closely match our estimation $\hat{m}_1^{(j)}, \hat{m}_2^{(j)}, \ldots, \hat{m}_{c_2\log N}^{(j)}$. Finally, we output $\hat{P}_{lmm} = \sum_j \mu_j$ as the estimate.

\end{enumerate}

\end{document}